\documentclass[11pt]{article}

\usepackage[margin=1in]{geometry}
\usepackage{amsmath,amssymb,amsthm,mathtools,bm,mathrsfs}
\usepackage{microtype}
\usepackage{graphicx}
\usepackage{subcaption}
\usepackage{xcolor}
\usepackage{hyperref}

\usepackage{subcaption}
\hypersetup{colorlinks=true,linkcolor=blue,citecolor=blue,urlcolor=blue}

\newtheorem{theorem}{Theorem}
\newtheorem{lemma}{Lemma}
\newtheorem{proposition}{Proposition}

\newtheorem{assumption}{Assumption}
\newtheorem{remark}{Remark}

\DeclareMathOperator{\tr}{tr}

\DeclareMathOperator{\supp}{spt}

\newcommand{\R}{\mathbb R}

\newcommand{\Sbb}{\mathbb S}

\newcommand{\revision}[1]{{\color{black} #1}}

\begin{document}
	
	\title{Semidefinite relaxations for nonlinear elasticity  with energies convex in the Cauchy-Green strain tensor}
    
	\author{Didier Henrion$^{1,2}$, Milan Korda$^{1,2}$, Martin Kru\v z\'{\i}k$^{3,4}$, Karol\'{\i}na Sehnalov\'a$^1$}

\footnotetext[1]{Faculty of Electrical Engineering, Czech Technical University in Prague, Technick\'a 2, CZ-16626 Prague, Czechia.}
\footnotetext[2]{CNRS; LAAS; Universit\'e de Toulouse, 7 avenue du colonel Roche, F-31400 Toulouse, France. }
\footnotetext[3]{Institute of Information Sciences and Automation, Czech Academy of Sciences, Pod Vod\'arenskou v\v e\v z\'{\i} 4, CZ-18200 Prague, Czechia.}
\footnotetext[4]{Faculty of Civil Engineering, Czech Technical University in Prague, Th\'{a}kurova 7, CZ-16626 Prague, Czechia.}
	\date{\today}
	\maketitle
	
	\begin{abstract}
		In nonlinear elasticity, finding the deformation of a material which minimizes
		a given stored energy density is a challenging calculus of variations problem which may fail to have minimizers: the energy optimal material forms infinitely fine
		microstructures (wrinkles) rather than deforming smoothly.
		In the case where the energy function is non-convex but frame indifferent and
		convex with respect to the Cauchy-Green strain tensor, we use the standard Le Dret-Raoult semidefinite
		projection formula for the quasiconvex envelope of the energy function \revision{together with a recent no gap result for convex calculus of variation problems} to prove
		that there is no relaxation gap between the original non-convex calculus of
		variations problem and its linear moment formulation based on occupation
		measures. This implies convergence of the Lasserre moment-sum-of-squares (SOS) hierarchy and provides a computationally efficient, mesh-free numerical method that, unlike the finite element method, avoids undesirable mesh-dependent artifacts. Under the additional condition that the boundary condition is linear and the function is SOS convex in the strain tensor, we show that the first
		relaxation of the Lasserre hierarchy is exact. In other words, computing the quasiconvex envelope at a point boils down to solving a
		small convex semidefinite optimization problem.
	\end{abstract}

	\section{Introduction}
    
	\subsection{Nonlinear elasticity and stored energy minimization}
	
\begin{figure}[ht]
	\centering
	\includegraphics[width=0.7\textwidth]{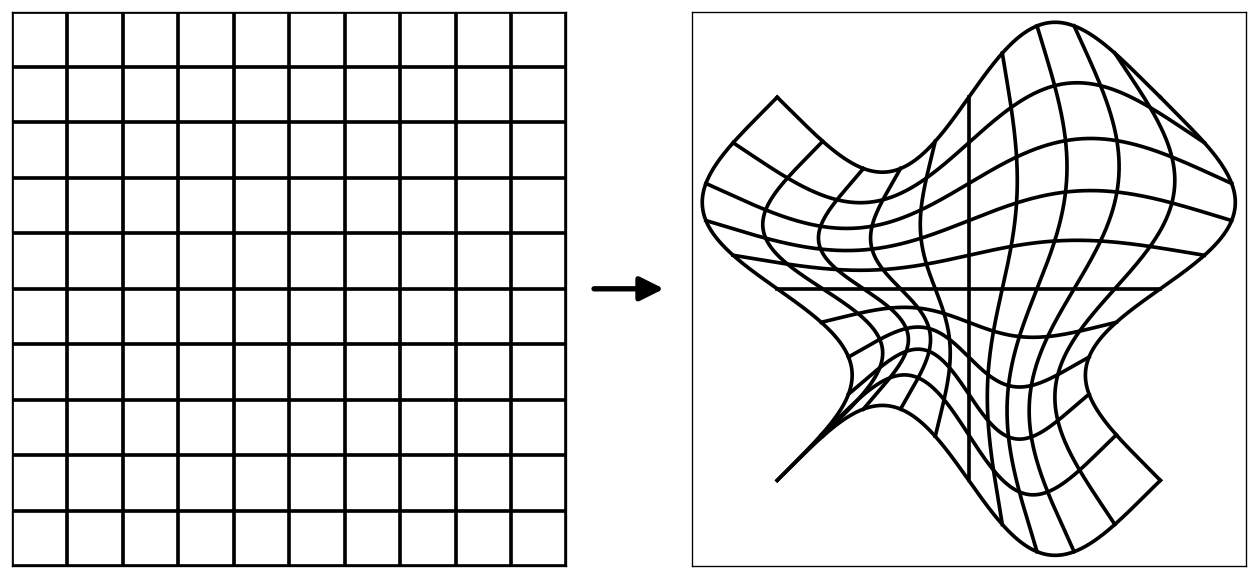}
	\caption{Left: reference configuration $\Omega$ with Cartesian grid.
		Right: deformed configuration $y(\Omega)$, image of the grid under
		the deformation $y$.}
	\label{fig:deformation}
\end{figure}
	
	Let $\Omega\subset\R^n$ be a bounded Lipschitz open domain, a reference configuration of an elastic body
	($n=2$ or $3$ in applications), and let
	\[
	y:\Omega\to\R^n
	\]
	denote a deformation mapping a material point $x\in\Omega$ in the reference
	configuration to its deformed position $y(x)$ in the current configuration.
	The  {deformation gradient} is the Jacobian matrix
	\[
	F(x):=\nabla y(x)\in\R^{n\times n}.
	\]
		Figure~\ref{fig:deformation} sketches the mapping $x\mapsto y(x)$ and the
	associated deformation of a Cartesian grid.

	In hyperelasticity, the elastic response is encoded in a stored energy density
	\[
	W:\R^{n\times n}\to[0,+\infty),
	\]
	and the total elastic energy associated with a deformation $y$ is
	\begin{equation}\label{eq:energy-functional}
		\int_\Omega W(\nabla y(x))\,dx.
	\end{equation}
    For a given $p>1$, let ${\mathscr W}^{1,p}_0(\Omega;\R^n)$ denote the Sobolev space of functions whose weak derivatives are in ${\mathscr L}^p(\Omega;\R^n)$ (the Lebesgue space of functions with integrable $p$-th powers) and which have zero trace on $\partial \Omega$. Let us prescribe Dirichlet conditions on the domain boundary $\partial \Omega$, and let $y_{\partial} \in {\mathscr W}^{1,\infty}(\Omega; \R^n)$  
    denote a given extension to $\Omega$. Accordingly, we denote the admissible function class by $y_{\partial}+{\mathscr W}^{1,p}_0(\Omega;\mathbb{R}^n)$.
    Therefore, we consider the
	calculus-of-variations problem
	\begin{equation}\label{eq:cov}
		\begin{aligned}
			J\ :=\ \inf_{y}\ &\int_\Omega W(\nabla y(x))\,dx\\
			\text{s.t. }&y\in y_{\partial}+{\mathscr W}^{1,p}_0(\Omega;\R^n).
		\end{aligned}
	\end{equation}	
    
	The functional \eqref{eq:cov} is typically non-convex in $\nabla y$,
	and minimizers may fail to exist even under smooth boundary data.
	A standard sufficient and  necessary  condition for the  existence of solutions under suitable growth conditions   is quasiconvexity of the
	integrand $W$ in the sense of Morrey \cite{Dacorogna2008}. Stronger conditions include
	polyconvexity and convexity, and a weaker condition is rank-one convexity, see e.g.
	\cite{Ciarlet1988,Dacorogna2008,KruzikRoubicek2019,Roubicek2020}.

	Given $F\in\R^{n\times n}$, the quasiconvex envelope of a continuous $W$  at $F$ is
	defined by
	\begin{equation}\label{eq:QW}
	     \begin{aligned}
			W_{\mathrm{quasi}}(F)\ :=\ \inf_y\ &\frac{1}{|\Omega|} \int_\Omega W(\nabla y(x))\,dx\\
			\text{s.t. }&y\in Fx+{\mathscr W}^{1,p}_0(\Omega;\R^n),
		\end{aligned}    
    \end{equation}
	where $|\Omega|$ is the Lebesgue volume of $\Omega$, see e.g. \cite{Dacorogna2008}. Up to normalization, it is the value of problem \eqref{eq:cov} for linear boundary conditions $y_{\partial}(x) = Fx$. It is the largest quasiconvex function below $W$. We say that a function is quasiconvex whenever it is everywhere equal to its quasiconvex envelope.
	
	Consider the relaxed
	calculus-of-variations problem
	\begin{equation}\label{eq:relaxed}
	\begin{aligned}
		J_{\mathrm{quasi}}\ :=\ \inf_y\ &\int_\Omega W_{\mathrm{quasi}}(\nabla y(x))\,dx\\
		\text{s.t. }&y\in y_{\partial}+{\mathscr W}^{1,p}_0(\Omega;\R^n)
	\end{aligned}
	\end{equation}
and the following standard growth conditions on the integrand.

\begin{assumption}[Two-sided growth]\label{ass:growth-two-sided}
	Let $p>1$ and let $W:\R^{n\times n}\to[0,+\infty)$ be continuous.
	We assume that there exist constants $c_1,c_2>0$ and $c_0\ge 0$ such that
	\begin{equation}\label{eq:two-sided-p-growth}
		c_1 |F|^p - c_0 \;\le\; W(F) \;\le\; c_2\bigl(1+|F|^p\bigr)
		\qquad\forall F\in\R^{n\times n}.
	\end{equation}
\end{assumption}

In the growth condition \eqref{eq:two-sided-p-growth},
$|F|$ denotes the Frobenius norm of $F$, i.e. the Euclidean norm of the vector of entries of matrix $F$.
The following result is standard in calculus of variations, see e.g. \cite[Thm. 9.1]{Dacorogna2008}.

\begin{theorem}\label{thm:M-equals-Mquasi}
Under Assumption~\ref{ass:growth-two-sided}, the original and relaxed problems have the same infimum
		\[
		J \;=\; J_{\mathrm{quasi}}
		\]
        and the infimum of the relaxed problem is attained.
		Moreover, for every $y\in y_\partial + \mathscr W^{1,p}_0(\Omega;\R^n)$ there exists
		a sequence $\{y_k\}\subset y_\partial + \mathscr W^{1,p}_0(\Omega;\R^n)$ such that
		\[
		y_k \to y \quad\text{in }{\mathscr L}^p(\Omega;\R^n),
		\qquad
		\int_\Omega W(\nabla y_k(x))\,dx
		\;\longrightarrow\;
		\int_\Omega W_{\mathrm{quasi}}(\nabla y(x))\,dx.
		\]
\end{theorem}

In practice, the main difficulty for using this theorem is that an explicit expression for the quasi-convex envelope \eqref{eq:QW} is rarely available. For convex envelopes, we have many tools (supporting hyperplanes, Legendre transforms, explicit formulas in low dimension), but for quasiconvex envelopes, there is no general representation theorem that turns the infinite-dimensional nonconvex variational problem \eqref{eq:QW} into a simple pointwise operation. Actually, solving problem \eqref{eq:QW}
is as difficult as solving the original problem \eqref{eq:cov}.

	\subsection{A model problem with microstructure}\label{sec:cauchygreen}
	
	For a simple example illustrating the difficulties faced in non-convex calculus of variations, consider the two-dimensional stored energy
	\begin{equation}\label{eq:cauchy-green-energy}
		W(F)=|F^T F-I_2|^2
	\end{equation}
	where  $I_2$ is the
	$2\times2$ identity matrix, and $|.|^2$ is the squared Frobenius norm, the sum of squares of the entries of a matrix. This is the squared distance, in the
	right Cauchy-Green strain tensor 
    \[C:=F^T F,\]
    from the natural
	state $I_2$. It is convex as a function of $C$ but not quasiconvex
	and even not rank-one convex as a function of $F$.
	
	Consider problem \eqref{eq:cov} with linear boundary data $y_\partial(x)=Ax$ with $A=\alpha I_2$ and $0<\alpha<1$. The energy of the homogeneous deformation $y(x)=Ax$ is strictly positive, since $W(A) = |(\alpha^2-1)I_2|^2
	= 2(\alpha^2-1)^2 > 0$ and
	$\int_\Omega W(A)\,dx = 2|\Omega|(\alpha^2-1)^2 > 0$.
	On the other hand, the wells of $W$ are precisely the orthogonal matrices $F \in O(2)$ at which $W(F)=0$. It can be proved that the quasiconvex envelope of $W$ is $W_{\mathrm{quasi}}(F) = 0$ whenever all singular values of $F$ are less than or equal to $1$.
	Since $A=\alpha I_2$ has singular values $\alpha\le 1$, we obtain $W_{\mathrm{quasi}}(A)=0$, and therefore, by the relaxation Theorem \ref{thm:M-equals-Mquasi}, $J=J_{\mathrm{quasi}}=
	\int_\Omega W_{\mathrm{quasi}}(A)\,dx
	\;=\; 0$.
	Hence
	\[
	J = \inf_{y} \int_\Omega W(\nabla y(x))\,dx \;=\; 0
	\;<\; \int_\Omega W(A)\,dx \;=\; 2|\Omega|(\alpha^2-1)^2,
	\]
	and the infimum is not attained in $y_\partial+\mathscr W^{1,p}_0(\Omega;\R^2)$. Minimizing sequences necessarily develop finer and finer oscillations in the interior of $\Omega$: the gradients $\nabla y_k$ take values close to the wells on most of the domain, so that $W(\nabla y_k)\approx 0$, while oscillating in such a way that their average gradient converges to $A$ and the boundary condition $y_k=Ax$ on $\partial\Omega$ is satisfied in the classical sense. This illustrates the formation of microstructures and the strict inequality between the non-relaxed energy and its quasiconvex relaxation.

	\subsection{Numerical difficulties and grid dependence}
	
	Finite element discretizations of non-quasiconvex problems such as the one of Section \ref{sec:cauchygreen} typically exhibit mesh-dependent patterns:
	the computed microstructure depends strongly on the mesh orientation
	and refinement; see, for example, the numerical experiments in \cite[Sec. 6.2, Fig. 3]{HorakKruzik2020}. Similar phenomena are
	widely reported in the numerical analysis literature on non-convex
	variational problems and phase transitions.
	
	Several numerical strategies attempt to deal with the lack of minimizers:
	\begin{itemize}
		\item introducing higher-order gradient regularization, e.g.\ via
		gradient polyconvex energies;
		\item direct numerical computation of quasiconvex envelopes using
		laminates or polyconvex/rank-one convex approximations;
		\item optimization directly over Young measures describing probability
		distributions of gradients.
	\end{itemize}
	Most approaches are mesh-based and require solving large nonlinear or
	nonsmooth problems on fine triangulations in order to resolve the
	microstructure.
	
\subsection{Convexity in the Cauchy-Green strain tensor}
	
	In this paper, our main assumptions on the stored energy $W$ are the following.

    \begin{assumption}[Frame indifference, or left invariance]\label{ass:frame}
		For all $F\in\R^{n\times n}$ and all $R\in O(n)$ (orthogonal matrices) we have
		\[
		W(RF)=W(F).
		\]
	\end{assumption}
	
	Assumption~\ref{ass:frame} is equivalent to the existence of a reduced energy function $\widetilde W:\Sbb^n_+\to[0,+\infty]$ such that
    \begin{equation}\label{eq:w_tilde}
        W(F)=\widetilde W(F^T F)=\widetilde W(C),
    \end{equation}
    see \cite{LeDretRaoult1995}.
	Physically, this means that the energy is invariant with respect to the (linear) change of coordinates. If the energy function is also right invariant, that is, $W(FR) = W(F)$ for $R\in O(n)$, then the energy function depends on $F$ only through the principal stretches (singular values) of the deformation gradient. While left invariance is a property of every energy function, right invariance is a property of isotropic elastic  materials. Accordingly,  right invariance is called isotropy.
	
	\begin{assumption}[Strain convexity]\label{ass:convex}
		The function $\widetilde W:\Sbb^n_+\to\R$ is convex.
	\end{assumption}

Under these regularity assumptions, Le Dret and Raoult \cite{LeDretRaoult1995} proved the following semidefinite projection
formula for the quasiconvex envelope. Let us define	\begin{equation}\label{eq:Wproj}
		W_{\mathrm{proj}}(F)
		:=\inf_{P\in\Sbb^n_+}\widetilde W(F^T F+P).
	\end{equation}

\begin{theorem}[Semidefinite projection formula]\label{thm:LDR}
If $W$ satisfies Assumptions~\ref{ass:frame}
	and~\ref{ass:convex} then
	\begin{equation}\label{eq:LDR}
		W_{\mathrm{quasi}}(F)\ =\ W_{\mathrm{proj}}(F)
	\end{equation}
		 for all $F\in\R^{n\times n}$.
	In particular, $W_{\mathrm{quasi}}$ coincides with the convex, polyconvex and rank-one
	convex envelopes of $W$.
\end{theorem}

Thus the problem of computing the quasiconvex envelope of $W$ at a given matrix $F$ reduces to
a convex optimization problem in the strain variable $C=F^T F$ and a
positive semidefinite slack variable $P$.  The stored energy densities that are convex in $F^\top F$  have interesting properties that relate the solutions of the equilibrium equations in elasticity and minimizers  of $J$, see \cite{Gao}.

	\subsection{Contribution}
	
This paper follows a recent line of research \cite{HenrionKordaKruzikRios2024,Chernyavsky2023,FantuzziTobasco2022,FantuzziFuentes2024} which promotes a different numerical method for non-convex calculus of variations, based on semidefinite optimization and the moment-SOS hierarchy a.k.a. Lasserre hierarchy \cite{HenrionKordaLasserre2020}, as a follow-up of previous research endeavors on optimal control of ordinary differential equations, see e.g. \cite{Claeys2017, Augier2024} and references therein. The key ideas are as follows.
\begin{enumerate}
	\item Instead of minimizing \eqref{eq:cov} over deformations $y$, we
	relax the problem to a linear program (LP) over measures. We show \revision{in Theorem \ref{thm:M-equals-Mlin}} that there is no relaxation gap, i.e. the LP on measures has the same infimum as \eqref{eq:cov} under the convexity Assumptions \ref{ass:frame} and \ref{ass:convex}. \revision{This follows by combining results of Le Dret and Raoult \cite{LeDretRaoult1995} with a recent no relaxation gap result for convex calculus of variations \cite{krzgap}.} This implies convergence of the Lasserre hierarchy for solving \eqref{eq:cov} with polynomial data.
	
	\item In the case of linear boundary condition (equivalent to computing the quasiconvex envelope at a given point),  and if the energy function is SOS convex in the strain tensor  (i.e. its Hessian is a polynomial matrix SOS, which implies convexity), we show \revision{in Theorem \ref{thm:first-relax-exact}} that the first relaxation in the Lasserre hierarchy is exact: solving \eqref{eq:cov} amounts to solving a small-dimensional semidefinite problem (SDP) in the strain tensor.  \revision{This follows by using an extended Jensen's inequality originally proposed for polynomial optimization in \cite[Thm. 2.6]{lasserre09}, see also \cite[Sec. 7.1.2]{nie23}.}

\end{enumerate}
From the computational viewpoint, this means that, for a broad class of
nonlinear elastic energies, solving globally a non-convex 
calculus of variations problem reduces to solving a family of mesh-free finite-dimensional SDPs, for which efficient solvers are available \cite{BenTalNemirovski2001}.
Moreover, computing the quasiconvex envelope pointwise amounts to solving a single low-dimensional SDP.
 
	\subsection{Outline}
	
	The rest of the paper is organized as follows. In
	Section~\ref{sec:occupation} we describe the linear formulation of the calculus of variations problem with occupation
	measures and prove the absence of relaxation gap. Section~\ref{sec:momsos}
	introduces the Lasserre hierarchy of semidefinite relaxations and proves its convergence. Section~\ref{sec:exact} shows the exactness of the
	first relaxation for computing the quasiconvex envelope, under the condition that the strain tensor is SOS convex. Section~\ref{sec:extraction} briefly discusses minimizer extraction and Section~\ref{sec:examples}
	illustrates numerically the approach on the isotropic Saint Venant-Kirchhoff energy density and an anisotropic variation.  We conclude in Section~\ref{sec:conclusion} with
	remarks and open questions.

\section{Occupation measures}\label{sec:occupation}

First, we rephrase the calculus of variations problem in terms of {occupation measures}, following the program outlined in \cite{KordaHenrionLasserre2022}.
Given a deformation $y\in y_\partial + \mathscr W^{1,p}_0(\Omega;\R^n)$, we introduce the
occupation measure $\mu$ on
$\Omega\times\R^n\times\R^{n\times n}$ by
\begin{equation}\label{eq:occupation}
\int_{\Omega\times\R^n\times\R^{n\times n}}\!\!\!
\phi(x,y,Z)\,d\mu(x,y,Z)
:=\int_\Omega \phi\bigl(x,y(x),\nabla y(x)\bigr)\,dx
\end{equation} 
for all bounded Borel test functions
$\phi:\Omega\times\R^n\times\R^{n\times n}\to\R$.  In other words, it is the push-forward or image measure of the uniform or Lebesgue measure on $\Omega$ through the map $x \mapsto (x,y(x),\nabla y(x))$.

We also introduce the boundary occupation measure $\mu_\partial$ on
$\partial\Omega\times\R^n$ by
\begin{equation}\label{eq:boundary}
\int_{\partial\Omega\times\R^n}\phi(x,y)\,d\mu_\partial(x,y)
:=\int_{\partial\Omega}\phi\bigl(x,y_\partial(x)\bigr)\,d\sigma(x)
\end{equation}
for all bounded Borel test functions $\phi : \partial \Omega \times \R^n \to \R$, where $d\sigma$ denotes the surface or Hausdorff
measure on $\partial\Omega$. It is the push-forward of the Hausdorff measure through the map $x \mapsto (x,y_{\partial}(x))$.

Let $\phi:\overline\Omega\times\R^n\to\R^n$ be a smooth vector field
with compact support in $\overline\Omega\times\R^n$. For a smooth
deformation $y$ the chain rule gives
\[
\partial_{x_i}(\phi_i(x,y(x)))
=
\partial_{x_i}\phi_i(x,y(x))
+\sum_{j=1}^n
\partial_{y_j}\phi_i(x,y(x))\,\partial_{x_i}y_j(x),
\]
and the divergence theorem yields
\[
\int_\Omega
\sum_{i=1}^n \left(\partial_{x_i}\phi_i(x,y(x))
+\sum_{j=1}^n
\partial_{y_j}\phi_i(x,y(x))\,\partial_{x_i}y_j(x)\right)dx
=
\int_{\partial\Omega}
\phi(x,y(x))\cdot n(x)\,d\sigma(x),
\]
where $n(x)$ is the outward unit normal on $\partial\Omega$.
In terms of the occupation measures $(\mu,\mu_{\partial})$ this identity can be written as the Stokes constraints 
\begin{equation}\label{eq:stokes-measure}	\int_{\Omega\times\R^n\times\R^{n\times n}}
\mathcal D\phi(x,y,Z)\,d\mu(x,y,Z)
	=
	\int_{\partial\Omega\times\R^n}
	\phi(x,y)\cdot n(x)\,d\mu_\partial(x,y)
\end{equation}
upon defining as in \cite{FantuzziTobasco2022} the total divergence operator
\begin{equation}\label{eq:div}
\mathcal D \phi(x,y,Z) := \sum_{i=1}^n\left(\partial_{x_i}\phi_i(x,y)
+\sum_{j=1}^n\partial_{y_j}\phi_i(x,y)\,Z_{j,i}\right).
\end{equation}
The infinite-dimensional linear programming (LP) relaxation of the calculus of variations problem \eqref{eq:cov} then reads 
\begin{equation}\label{eq:LP-occupation}
	\begin{aligned}
		J_{\mathrm{lin}}:= \inf_{\mu}\quad
		&\int_{\Omega\times\R^n\times\R^{n\times n}} W(Z)\,d\mu(x,y,Z) \\
		\text{s.t.}\quad
		& 	\int_{\Omega\times\R^n\times\R^{n\times n}} 
		\mathcal D\phi(x,y,Z) \, d\mu(x,y,Z) 
		=
\int_{\partial\Omega\times\R^n}\phi(x,y)\cdot n(x)\,d\mu_\partial(x,y), \quad \forall \phi \in \Phi
	\end{aligned}
\end{equation}
where $\Phi$ denotes the class of smooth test fields
$\varphi:\overline\Omega\times\mathbb R^n\to\mathbb R^n$
satisfying
\[
|\partial_x \phi(x,y)| \le C\,(1+|y|^p),
\qquad
|\partial_y \phi(x,y)| \le C\,(1+|y|^{p-1})
\]
for some constant $C$ and all $(x,y)\in\overline\Omega\times\mathbb R^n$.
This class contains in particular test fields that are constant or affine in $y$,
which will be used below.

In LP \eqref{eq:LP-occupation} the unknown is the measure $\mu$. The boundary measure $\mu_\partial$ is entirely specified by \eqref{eq:boundary} from the boundary condition.
  
\begin{theorem}[No relaxation gap]\label{thm:M-equals-Mlin}
Under Assumptions~\ref{ass:growth-two-sided}, \ref{ass:frame} and \ref{ass:convex}, the optimal value  
	of the occupation measure LP  \eqref{eq:LP-occupation} coincides
	with the value  of the original calculus-of-variations problem
	\eqref{eq:cov}, i.e.
	\[
	J_{\mathrm{lin}} \;=\; J.
	\]
\end{theorem}

\begin{proof}
Let us first prove that $J_{\mathrm{lin}}\le J$.
	Let $y\in y_\partial+\mathscr W^{1,p}_0(\Omega;\R^n)$ be admissible for
	\eqref{eq:cov}. Its associated occupation measure $\mu$ is defined as in \eqref{eq:occupation} and its associated boundary measure is defined as in \eqref{eq:boundary}.
	By construction, the measure $\mu$ satisfies the Stokes
	constraints \eqref{eq:stokes-measure}. Moreover,
	\[
	\int_{\Omega\times\R^n\times\R^{n\times n}} W(Z)\,d\mu(x,y,Z)
	= \int_\Omega W(\nabla y(x))\,dx.
	\]
	Thus every admissible deformation $y$ for \eqref{eq:cov} produces a
	feasible measure $\mu$ for \eqref{eq:LP-occupation} with the same
	objective value, and therefore
	\[
	J_{\mathrm{lin}} \;\le\; J.
	\]

Now let us prove the converse inequality $J_{\mathrm{lin}}\ge J$. Since the quasi-convex envelope $W_\mathrm{quasi}$ is convex by Theorem~\ref{thm:LDR} and the constraints are linear, it follows from~\cite[Theorem 3.1]{HenrionKordaKruzikRios2024} that the infimum in~\eqref{eq:LP-occupation} is equal to $J_\mathrm{quasi}$ when $W$ is replaced by $W_\mathrm{quasi}$. Since the optimal value of \eqref{eq:LP-occupation} is monotone with respect to the minimized potential, the conclusion that $J_\mathrm{lin} \ge J$  follows from the facts that $W \ge W_{\mathrm{quasi}}$ and $J = J_{\mathrm{quasi}}$ by Theorem~\ref{thm:M-equals-Mquasi}.

For the reader's convenience, and to keep this paper self-contained, we provide below a full proof of the converse inequality, by extracting the relevant technical devices from  the proof of ~\cite[Theorem 3.1]{HenrionKordaKruzikRios2024}. Let $\mu$ be any feasible measure for
\eqref{eq:LP-occupation} with finite energy
$\int W\,d\mu<\infty$.
We first disintegrate $\mu$ with respect to
its $x$-marginal and identify the corresponding barycentric
fields, following the arguments of \cite{HenrionKordaKruzikRios2024}.

By testing \eqref{eq:stokes-measure} with functions of the form
$\phi(x,y)=\psi(x)e_i$ for 
$\psi$ smooth and
summing over $i$, one shows that the $x$-marginal of $\mu$ is Lebesgue
measure on $\Omega$.
Hence there exists a family of probability measures
$\{\nu_x\}_{x\in\Omega}$ on $\R^n\times\R^{n\times n}$ such that
\[
d\mu(x,y,Z) = d\nu_x(y,Z)\,dx
\]
from the disintegration theorem. Define the barycentric fields
\begin{align} \label{eq:y_bar}
\bar y(x) := \int_{\R^n\times\R^{n\times n}} y\,d\nu_x(y,Z),\qquad
\bar Z(x) := \int_{\R^n\times\R^{n\times n}} Z\,d\nu_x(y,Z).    
\end{align}

Using the $p$-growth of $W$ and Assumption~\ref{ass:growth-two-sided},
$\bar y$ and $\bar Z$ belong to ${\mathscr L}^p(\Omega)$ by Jensen's
inequality and the integrability of $y$ and $Z$ with respect to $\mu$.

Next, we test \eqref{eq:stokes-measure} with
$\phi(x,y)=\psi(x)e_i$ 
for $\psi$ smooth.
A standard integration by parts argument, together with the definition
of $\mu_\partial$ as the pushforward of $\sigma$ by
$x\mapsto(x,y_\partial(x))$, shows that
\[
\bar y\in y_\partial+\mathscr W^{1,p}_0(\Omega;\R^n).
\]
Testing \eqref{eq:stokes-measure} with
$\phi(x,y)=\psi(x)y_j e_i$ and using the disintegration of $\mu$, one
obtains 
\[
\int_\Omega \partial_{x_i}\psi(x)\,\bar y_j(x)\,dx
+ \int_\Omega \psi(x)\,\bar Z_{j,i}(x)\,dx
=
\int_{\partial\Omega} \psi(x)\,y_{\partial,j}(x)\,n_i(x)\,d\sigma(x).
\]
Comparing this with the weak formulation of the identity
$\partial_{x_i}\bar y_j = \bar Z_{j,i}$ and recalling that the trace of
$\bar y$ on $\partial\Omega$ is $y_\partial$, we conclude that
$\bar y\in\mathscr W^{1,p}(\Omega;\R^n)$ with
\[
\nabla\bar y(x) = \bar Z(x)
\quad\text{for a.e. }x\in\Omega.
\]

We now compare the energy of $\mu$ with the quasiconvex envelope of $W$
evaluated at $\nabla\bar y$. Let $\zeta_x$ be the marginal of $\nu_x$
on $\R^{n\times n}$, so that
\[
\bar Z(x) = \int_{\R^{n\times n}} Z\,d\zeta_x(Z)
\quad\text{for a.e. }x\in\Omega.
\]

By Theorem~\ref{thm:LDR}, under Assumptions~\ref{ass:frame}
and~\ref{ass:convex} the quasiconvex envelope $W_{\mathrm{quasi}}$
coincides with the convex envelope of $W$. In particular, we have the
representation
\begin{equation}\label{eq:conv-envelope-repr}
	W_{\mathrm{quasi}}(F)
	=
	\sup\{V(F): V \text{ affine on }\R^{n\times n},\ V\leq W\}.
\end{equation}
Fix $x$ and apply this to $F=\bar Z(x)$. Using that $\bar Z(x)$ is the
barycenter of $\zeta_x$ and the linearity of affine maps, we obtain
\begin{align*}
	W_{\mathrm{quasi}}(\bar Z(x))
	&=
	\sup_{V\le W} V(\bar Z(x))
	=
	\sup_{V\le W}
	\int_{\R^{n\times n}} V(Z)\,d\zeta_x(Z)\\
	&\le
	\sup_{V\le W}
	\int_{\R^{n\times n}} W(Z)\,d\zeta_x(Z)
	=
	\int_{\R^{n\times n}} W(Z)\,d\zeta_x(Z),
\end{align*}
since  $V(Z)\le W(Z)$ for all $Z$.

Integrating this inequality over $\Omega$ and using the disintegration
of $\mu$ gives
\[
\int_\Omega W_{\mathrm{quasi}}(\nabla\bar y(x))\,dx
=
\int_\Omega W_{\mathrm{quasi}}(\bar Z(x))\,dx
\le
\int_\Omega\int_{\R^{n\times n}} W(Z)\,d\zeta_x(Z)\,dx
=
\int W\,d\mu.
\]
Since $\mu$ was an arbitrary feasible measure for
\eqref{eq:LP-occupation}, we obtain
\[
\inf_{\mu\ \text{feasible}}
\int_{\Omega \times \R^n \times \R^{n\times n}} W\,d\mu(x,y,Z)
\;\ge\;
\inf_{y\in y_\partial+\mathscr W^{1,p}_0(\Omega;\R^n)}
\int_\Omega W_{\mathrm{quasi}}(\nabla y(x))\,dx,
\]
that is,
\[
J_{\mathrm{lin}}
\;\ge\;
J_{\mathrm{quasi}}.
\]

Finally, by Theorem~\ref{thm:M-equals-Mquasi} we have $J=J_{\mathrm{quasi}}$, hence
\[
J_{\mathrm{lin}}
\;\ge\;
J.
\]
\end{proof}

\revision{
The proof of Theorem~\ref{thm:M-equals-Mlin} reveals that the essential
property of the stored energy is not convexity in the Cauchy-Green
strain tensor itself, but rather the property that the quasi convex envelope and the convex envelope coincide. We record this observation in the following
more general statement.

\begin{proposition}[General no-gap criterion]\label{prop:general-no-gap}
Assume that $W$ satisfies Assumption~\ref{ass:growth-two-sided} and that
its quasiconvex and convex envelopes coincide, i.e.,
\[
W_{\mathrm{quasi}}=W_{\mathrm{conv}}
\qquad\text{on }\R^{n\times n},
\]
where $W_{\mathrm{conv}}$ denotes the lower semicontinuous convex envelope of $W$.
Then the occupation-measure relaxation \eqref{eq:LP-occupation} has no
relaxation gap:
\[
J_{\mathrm{lin}}=J.
\]
\end{proposition}

\begin{proof}
The proof is the same as that of
Theorem~\ref{thm:M-equals-Mlin}, except that
Assumptions~\ref{ass:frame} and~\ref{ass:convex} are no longer needed.

Indeed, the inequality $J_{\mathrm{lin}}\le J$ follows, as before, by
associating with every admissible deformation its occupation measure.
For the converse inequality, repeat the disintegration and barycentric
construction in the proof of Theorem~\ref{thm:M-equals-Mlin}, yielding
an admissible deformation $\bar y$ such that
\[
\nabla\bar y(x)=\bar Z(x)
\qquad\text{for a.e. }x\in\Omega.
\]
By the assumption $W_{\mathrm{quasi}}=W_{\mathrm{quasi}}$ and the
Fenchel-Moreau representation of the convex envelope,
\[
W_{\mathrm{quasi}}(F)
=
\sup\{V(F):V\text{ affine on }\R^{n\times n},\ V\le W\}.
\]
Therefore the same barycentric argument as in
\eqref{eq:conv-envelope-repr} gives
\[
W_{\mathrm{quasi}}(\bar Z(x))
\le
\int_{\R^{n\times n}}W(Z)\,d\zeta_x(Z)
\qquad\text{for a.e. }x\in\Omega.
\]
After integration,
\[
J_{\mathrm{lin}}\ge J_{\mathrm{quasi}}.
\]
Finally, Theorem~\ref{thm:M-equals-Mquasi} gives
$J_{\mathrm{quasi}}=J$, and hence $J_{\mathrm{lin}}\ge J$.
Together with the opposite inequality this proves
$J_{\mathrm{lin}}=J$.
\end{proof}
}

\revision{
\begin{remark}[Failure of the projection formula without strain convexity]\label{rem:projfail}
The projection formula \eqref{eq:LDR} may fail when
Assumption~\ref{ass:convex} is dropped, even when the quasi convex envelope and the convex envelope coincide.
Let $n=1$, so that frame indifference means
that $W$ is even, and quasiconvexity reduces to convexity, whence $W_{\mathrm{quasi}}=W_{\mathrm{conv}}$
automatically. Let
$\widetilde W(c):=\tfrac13 c^3-\tfrac32 c^2+3c$ and
$W(F):=\widetilde W(F^2)$, which satisfies
Assumption~\ref{ass:growth-two-sided} with $p=6$. Since
$\widetilde W'(c)=(c-\tfrac32)^2+\tfrac34>0$, the function $\widetilde W$
is strictly increasing, so $W_{\mathrm{proj}}=W$. However $W$ is not
convex, since $W''(F)=10F^4-18F^2+6<0$ for
$F^2\in\bigl(\tfrac{9-\sqrt{21}}{10},\tfrac{9+\sqrt{21}}{10}\bigr)$, so that
\[W_{\mathrm{quasi}}=W_{\mathrm{conv}}<W=W_{\mathrm{proj}}\] on a nonempty interval.
Note that $\widetilde W''(c)=2c-3<0$ near the origin, so
Assumption~\ref{ass:convex} indeed fails.
\end{remark}
}

\revision{
\begin{remark}[Strain convexity is not necessary]\label{rem:notnecessary}
Assumption~\ref{ass:convex} is sufficient but not necessary for the
condition $W_{\mathrm{quasi}}=W_{\mathrm{conv}}$, and not necessary either for the
validity of the projection formula \eqref{eq:LDR}. Consider
$W(F)=|F|+|F|^p$, $p\ge 2$,
which satisfies Assumption~\ref{ass:growth-two-sided} with exponent $p$,
since $|F|\le 1+|F|^p$, and which is frame indifferent, and even
isotropic, since the Frobenius norm is invariant under left and right
orthogonal multiplication. It can be written as
$W(F)=\widetilde W(F^T F)$ with
$\widetilde W(C)=(\tr C)^{1/2}+(\tr C)^{p/2}$,
$C\in\Sbb^n_+.$
The reduced energy $\widetilde W$ is not convex: along the ray
$C=tI_n$, $t> 0$, the second derivative of
$t\mapsto (nt)^{1/2}+(nt)^{p/2}$ tends to $-\infty$ as $t\to 0^+$, so
Assumption~\ref{ass:convex} fails. Nevertheless, $W$ is convex in $F$,
as the composition of the Frobenius norm with the nondecreasing convex
function $t\mapsto t+t^p$, hence quasiconvex, so that
\[
W_{\mathrm{quasi}}=W=W_{\mathrm{conv}}.
\]
Moreover, $\widetilde W$ is nondecreasing with respect to the Loewner
order, since $C\preceq C'$ implies $\tr C\le\tr C'$, so the infimum in
\eqref{eq:Wproj} is attained at $P=0$ and
$W_{\mathrm{proj}}=W=W_{\mathrm{quasi}}$ as well. Together with
Remark~\ref{rem:projfail}, this shows that when
Assumption~\ref{ass:convex} fails, the identity
$W_{\mathrm{quasi}}=W_{\mathrm{conv}}$ may still hold, while the projection
formula \eqref{eq:LDR} may or may not remain valid.
\end{remark}
}

\section{Lasserre hierarchy}\label{sec:momsos}

We now assume that the stored energy density $W$ is a polynomial of degree at most $2d$ and
that the space of admissible triples $(x,y,Z)$ is a basic semialgebraic
set. More precisely, we assume that there exist polynomials
\[
g_1,\dots,g_m\in\R[x,y,Z]
\]
such that the support of the occupation measure $\mu$ is contained in a semialgebraic set
\[
K
:=
\Bigl\{(x,y,Z)\in\R^{n}\times\R^n\times\R^{n\times n}:
g_j(x,y,Z)\ge 0,\ j=1,\dots,m\Bigr\}.
\]
Under these assumptions the linear measure problem
\eqref{eq:LP-occupation} is a generalized moment problem, and it can be approximated by a hierarchy of semidefinite
relaxations, the moment-SOS a.k.a. Lasserre hierarchy \cite{HenrionKordaLasserre2020}.

\subsubsection*{Truncated moment sequences and moment matrices}

Given $r \geq d$, let $\mathcal B_r=(b_\alpha(x,y,Z))_{|\alpha|\le 2r}$ denote a basis spanning the vector space of polynomials
\[
p(x,y,Z) = \sum_{|\alpha|\leq 2r} p_\alpha b_\alpha(x,y,Z)
\]
of total degree at most $2r$, and let $z=\{z_\alpha\}_{|\alpha|\le 2r}$
be a truncated moment sequence indexed by multi-indices
$\alpha\in\mathbb N^{n+n+n^2}$, where
\[
z_\alpha
=
\int_{K}
b_{\alpha}(x,y,Z) \,d\mu(x,y,Z),
\]
and $\alpha = (\alpha_x,\alpha_y,\alpha_Z)$ is split according to the
variables. Given a vector $z$, build the linear functional
\[
\ell_z\left(p(x,y,Z)\right) = \ell_z\left(\sum_\alpha p_\alpha b_\alpha(x,y,Z)\right)
:=\sum_\alpha p_\alpha z_\alpha.
\]
The {moment matrix} $M_r(z)$ is the symmetric matrix associated to the quadratic form $p  \mapsto \ell_z(p^2)$, where $p$ is a polynomial of degree at most $r$.
Positivity of the underlying measure $\mu$ implies
$M_r(z)\succeq 0$.

For each constraint polynomial $g_j$ we also introduce the associated
localizing matrix associated to the quadratic form $p  \mapsto \ell_z(g_j p^2 )$ where $p$ is a polynomial of degree at most $r-r_j$ with 
$r_j:=\left\lceil \frac{\deg g_j}{2}\right\rceil$.
If $\mu$ is supported on $K$ then $g_j\ge 0$ on $\supp\mu$,
and this implies $M_{r-r_j}(g_j z)\succeq 0$.
With the notation $g_0 \equiv 1$, $r_0:=0$,
the localizing matrix $M_{r-r_0}(g_0 y)$ is equal to the moment matrix $M_r(y)$.

\subsubsection*{Polynomial approximation of the Stokes constraints}

The Stokes constraints \eqref{eq:stokes-measure} read
\[
\int \mathcal D\phi\,d\mu
=
\int \phi\cdot n\,d\mu_\partial
\qquad\forall\phi\in\Phi,
\]
with $\mathcal D$ defined in \eqref{eq:div}.
To obtain a finite-dimensional relaxation we use finite set $\mathcal B_r$ of test functions spanning the vector space of polynomials of degree at most $2r$. Later on, in order to prove convergence of the relaxations when the degree $r$ goes to infinity, we will restrict the support of the polynomials to a compact set. 

The total divergence $\mathcal D\phi(x,y,Z)$ is a polynomial of degree up to $2r$, and we can write
\[
\int_{K} \mathcal D\phi(x,y,Z)\,d\mu(x,y,Z)
=
\ell_z\bigl(\mathcal D\phi\bigr),
\]
a linear form in the truncated moment sequence $z$.

On the right-hand side, the boundary measure $\mu_\partial$ is known
explicitly from the boundary data via \eqref{eq:boundary}. For any
polynomial vector field $\phi$ we can therefore precompute (or
symbolically represent) the boundary functional
\[
b(\phi)
:=
\int_{\partial\Omega\times\R^n} \phi(x,y)\cdot n(x)\,d\mu_\partial(x,y)
=
\int_{\partial\Omega} \phi\bigl(x,y_\partial(x)\bigr)\cdot n(x)\,d\sigma(x),
\]
which is a scalar depending linearly on $\phi$.

We thus
obtain a finite family of linear constraints on $z$:
\begin{equation}\label{eq:stokes-moment}
	\ell_z\bigl(\mathcal D \phi\bigr)
	=
	b(\phi),
	\qquad \phi \in \mathcal B_r.
\end{equation}

\subsubsection*{Moment relaxation}

The $r$-th order moment relaxation of the measure LP
\eqref{eq:LP-occupation} is the semidefinite program
\begin{equation}\label{eq:SDP-occupation}
	\begin{aligned}
		J_{\mathrm{mom}}^r
		:=\inf_{z}\quad & \ell_z(W)\\[0.3em]
		\text{s.t.}\quad
		& M_{r-r_j}(g_j z)\succeq 0,\qquad j=0,1,\dots,m\\
		& 	\ell_z\bigl(\mathcal D\phi\bigr)
		=
		b(\phi),
		\qquad \phi \in \mathcal B_r.
	\end{aligned}
\end{equation}

In our setting, the natural support of the occupation measure lies in
$\Omega\times\R^n\times\R^{n\times n}$, which is noncompact in the
$(y,Z)$ variables. To apply the moment-SOS machinery we therefore
introduce a truncation parameter $R>0$ and work on the compact basic
semialgebraic sets
\[
K_R
:=
K \cap
\bigl\{(x,y,Z)\in\overline\Omega\times\R^n\times\R^{n\times n}:
|y|^2+|Z|^2 \le R^2\bigr\}.
\]

By construction, the minimal relaxation order is
\[
r_{\min}:=\max\{\lceil\frac{\text{deg}\,W}{2}\rceil,\max_{j=0,1,\ldots,m}\lceil\frac{\text{deg}\,g_j}{2}\rceil\}.
\]
Let $J_{\mathrm{lin},R}$ denote the optimal value of the
occupation-measure LP \eqref{eq:LP-occupation} with the additional
support constraint $\supp\mu\subset K_R$, and let
$J^r_{\mathrm{mom},R}$ be the optimal value of the $r$-th order semidefinite
relaxation built on $K_R$. For each fixed $R$, the set $K_R$ is compact
and Archimedean by construction, and Putinar's Positivstellensatz
implies
\[
J^r_{\mathrm{mom},R}\ \uparrow\ J_{\mathrm{lin},R}
\qquad\text{as }r\to\infty
\]
as in the original Lasserre hierarchy for polynomial optimization \cite{Lasserre2001}.
By construction
$K_R\subset K_{R'}$ for $R\le R'$  
so the feasible sets are nested and
\[
J_{\mathrm{lin},R_1}\ \ge\ J_{\mathrm{lin},R_2}\ \ge\ J_{\mathrm{lin}}
\qquad\text{whenever }R_1\le R_2,
\]
i.e.\ $\{J_{\mathrm{lin},R}\}_R$ is a monotone decreasing sequence
bounded below by $J_{\mathrm{lin}}$. 
Let us now prove that it converges.

\begin{lemma}\label{lem:limr}
\revision{Under Assumptions~\ref{ass:growth-two-sided}, \ref{ass:frame} and \ref{ass:convex},} it holds $\lim_{R\to\infty} J_{\mathrm{lin},R}=J_{\mathrm{lin}}.$
\end{lemma}

\begin{proof}
\revision{
By Theorem~\ref{thm:M-equals-Mlin}, we have
$J_{\mathrm{lin}}=J$.
Since the support constraint defining $J_{\mathrm{lin},R}$ restricts
the feasible set of the occupation measure problem,
\[
J\;=\;J_{\mathrm{lin}}\;\le\;J_{\mathrm{lin},R}.
\]

For $R>0$, let
\[
\mathcal Y_R
:=
\Bigl\{
y\in y_\partial+\mathscr W^{1,p}_0(\Omega;\R^n):
|y(x)|^2+|\nabla y(x)|^2\le R^2
\ \text{for a.e. }x\in\Omega
\Bigr\}
\]
and
\[
J_R
:=
\inf_{y\in\mathcal Y_R}
\int_\Omega W(\nabla y(x))\,dx.
\]
Every $y\in\mathcal Y_R$ generates an occupation measure supported on
$K_R$ with the same objective value. Hence
\[
J_{\mathrm{lin},R}\le J_R.
\]
It therefore suffices to prove that $J_R\to J$.
}
   
Fix $\varepsilon>0$. By definition of $J$, there exists
$y\in y_{\partial}+\mathscr W^{1,p}_0(\Omega;\mathbb{R}^n)$
such that
\[
\int_{\Omega} W(\nabla y(x))\,dx \le J+\frac{\varepsilon}{2}.
\]
\revision{Recall that $y_\partial \in \mathscr W^{1, \infty}(\Omega; \R^n)$.} Define $u:=y-y_{\partial}\in \mathscr W^{1,p}_0(\Omega;\mathbb{R}^n).$
\revision{By definition of $\mathscr W^{1,p}_0(\Omega;\R^n)$ as the closure of
$\mathscr C^\infty_c(\Omega;\R^n)$ in the $\mathscr W^{1,p}$ norm, there exists a
sequence $u_k\in \mathscr C^\infty_c(\Omega;\R^n)$ such that $u_k\to u$ in
$\mathscr W^{1,p}(\Omega;\R^n)$.}
Define $y_k:=y_{\partial}+u_k$.
\revision{
Then $y_k\in y_{\partial}+\mathscr C^\infty_c(\Omega;\R^n)
\subset y_{\partial}+\mathscr W^{1,p}_0(\Omega;\R^n)$,}
so each $y_k$ has the same trace on $\partial\Omega$ as $y_{\partial}$, and $y_k\to y$ in $\mathscr W^{1,p}(\Omega;\mathbb{R}^n)$.

Passing to a subsequence if necessary, we may assume that $\nabla y_k(x)\to \nabla y(x)$ for a.e. $x\in\Omega.$ \revision{Indeed, this follows from the strong convergence $u_k \to u$.} Since $W$ is continuous, it follows that
\[W(\nabla y_k(x))\to W(\nabla y(x))\]
for a.e. $x\in\Omega$. Moreover, since $\nabla y_k\to \nabla y$ in $\mathscr L^p(\Omega;\mathbb R^{n\times n})$, the sequence $\{|\nabla y_k|\}_k$ converges to $|\nabla y|$ almost everywhere and in
$\mathscr L^p(\Omega)$. By Vitali's Theorem
\cite[Thm.~B.2.4]{KruzikRoubicek2019}, the family
$\{|\nabla y_k|^p\}_k$ is uniformly integrable.
\revision{By Assumption \ref{ass:growth-two-sided},} we have $|W(\nabla y_k)|\le C\bigl(1+|\nabla y_k|^p\bigr)$
for some constant $C>0$, and therefore the family
$\{|W(\nabla y_k)|\}_k$ is uniformly integrable as well.
Applying again Vitali's Theorem, now with
$p=1$ to the sequence $W(\nabla y_k)$, we obtain
\[W(\nabla y_k)\to W(\nabla y)\quad\text{ in }\quad \mathscr L^1(\Omega),\]
and thus
\[
\int_\Omega W(\nabla y_k(x))\,dx \to \int_\Omega W(\nabla y(x))\,dx.
\]
\revision{As a side remark, since every admissible deformation in
$y_\partial+\mathscr W^{1,p}_0(\Omega;\R^n)$ can thus be approximated in energy
by admissible deformations in
$(y_\partial+\mathscr W^{1,p}_0(\Omega;\R^n))\cap \mathscr W^{1,\infty}(\Omega;\R^n)$,
there is no Lavrentiev gap between these two admissible classes, i.e.,
\[
\inf_{y\in y_\partial+\mathscr W^{1,p}_0(\Omega;\R^n)}
\int_\Omega W(\nabla y)\,dx
=
\inf_{\substack{y\in y_\partial+\mathscr W^{1,p}_0(\Omega;\R^n)\\
                y\in \mathscr W^{1,\infty}(\Omega;\R^n)}}
\int_\Omega W(\nabla y)\,dx.
\]}
Now choose $k$ large enough so that
\[
\left|\int_{\Omega} W(\nabla y_k(x))\,dx
-
\int_{\Omega} W(\nabla y(x))\,dx\right|
\le \frac{\varepsilon}{2}.
\]
Then
\[\int_{\Omega} W(\nabla y_k(x))\,dx \le J+\varepsilon.\]
\revision{Since $y_\partial\in\mathscr W^{1,\infty}(\Omega;\R^n)$ and
$u_k\in\mathscr C^\infty_c(\Omega;\R^n)$, the deformation
$y_k=y_\partial+u_k$ belongs to $\mathscr W^{1,\infty}(\Omega;\R^n)$, and
since $\Omega$ is bounded, the quantity
$R^2_k:=\||y_k|^2+|\nabla y_k|^2\|_{{\mathscr L}^\infty(\Omega)}$
is finite. Consequently $y_k\in\mathcal Y_R$ for every $R\ge R_k$, and hence
\[
J
\le J_{\mathrm{lin},R}
\le J_R
\le \int_\Omega W(\nabla y_k(x))\,dx
\le J+\varepsilon
\qquad\text{for all } R\ge R_k.
\]
Since $\varepsilon>0$ is arbitrary and $R\mapsto J_{\mathrm{lin},R}$ is
nonincreasing, we conclude
\[
\lim_{R\to\infty}J_{\mathrm{lin},R}
=J
=J_{\mathrm{lin}}.
\]}
\end{proof}

The dual of \eqref{eq:SDP-occupation} is an SOS problem that can be
naturally interpreted as a polynomial approximation of the dual
occupation-measure LP. We do not describe it here. The moment-SOS hierarchy  
provides a sequence of finite-dimensional SDPs whose optimal values
converge to the optimal value of the linear occupation-measure problem
\eqref{eq:LP-occupation}, and whose dual solutions yield polynomial
subsolutions of the Stokes inequality.

\revision{\revision{
\begin{remark}[Truncation and certified lower bounds]\label{rem:lowerbounds}
A hallmark of moment-SOS relaxations is that they bound the minimum from
below, in contrast with finite element methods, which produce upper
bounds. The truncation to $K_R$ partially compromises this feature:
for fixed $R$, the values $J^r_{\mathrm{mom},R}$ increase towards
$J_{\mathrm{lin},R}$, which dominates $J$ and coincides with it only
in the limit $R\to\infty$ by Lemma~\ref{lem:limr}. Each individual
value $J^r_{\mathrm{mom},R}$ therefore certifies a lower bound on the
truncated value $J_{\mathrm{lin},R}$, not on $J$, and $J$ is recovered
only in the double limit $r,R\to\infty$. Note that for fixed $r$ the
map $R\mapsto J^r_{\mathrm{mom},R}$ is nonincreasing, since for
$R'\ge R$ the ball localizing constraint of radius $R'$ is implied by
that of radius $R$, i.e.
$M_{r-1}\bigl((R'^2-|y|^2-|Z|^2)\,z\bigr)
=
M_{r-1}\bigl((R^2-|y|^2-|Z|^2)\,z\bigr)
+ (R'^2-R^2)\,M_{r-1}(z)
\;\succeq\; 0$.
This justifies the practical rule of increasing $R$ until the computed value
stabilizes.

The lower bound property can be preserved by simply removing the ball
constraint, since weak duality does not require compactness. Let
$J^r_{\mathrm{mom}}$ denote the value of the $r$-th relaxation
\eqref{eq:SDP-occupation} without the ball constraint, so that
$J^r_{\mathrm{mom}} \le J^r_{\mathrm{mom},R}$ for every $R$. The
near-optimal deformations
$y_k \in y_\partial + \mathscr C^\infty_c(\Omega;\R^n)$ constructed in
the proof of Lemma~\ref{lem:limr} have occupation measures with
compact support, hence with finite moments of all orders, and these
moments are feasible for \eqref{eq:SDP-occupation}, provided that, as
in all our examples, the polynomials $g_j$ describe $\overline\Omega$
and leave the variables $(y,Z)$ unconstrained. Consequently
$J^r_{\mathrm{mom}} \;\le\; J$ for every $r\ge r_{\min}$,
and the sequence $r\mapsto J^r_{\mathrm{mom}}$ is monotone
nondecreasing. These certified bounds can be tightened without losing
validity by appending the linear moment constraint
$\ell_z\bigl(|y|^p+|Z|^p\bigr)\;\le\;\gamma$. Indeed,
Assumption~\ref{ass:growth-two-sided} forces a polynomial $W$ to have
even degree $p=2d$, so that $|y|^p$ and $|Z|^p$ are indeed
polynomials. By coercivity, every admissible deformation with energy
at most $\beta+1$, where
$\beta:=\int_\Omega W(\nabla y_\partial)\,dx\ge J$, satisfies
$\int_\Omega|\nabla y|^p\,dx \le (\beta+1+c_0|\Omega|)/c_1$, and the
Poincar\'e inequality then bounds $\int_\Omega|y|^p\,dx$ in terms of
the same quantities and $\|y_\partial\|_{\mathscr W^{1,\infty}}$.
Choosing $\gamma$ accordingly, the constraint is satisfied by the
deformations $y_k$ for $k$ large, so the strengthened lionear problem still has
value $J$ and every relaxation value remains a certified lower bound
on $J$.

What is lost without truncation is the proof of convergence of the moment-SOS relaxation, since
Putinar's Positivstellensatz requires the Archimedean property. Coercivity of the stored energy, conic duality and homogenization techniques
for non-compact sets  are natural ingredients for a
proof of convergence without truncation. This was achieved for optimal control of ordinary differential equations in our previous work \cite{Sehnalova2025}. The extension of these techniques to calculus of variations is left for further work.
\end{remark}
}}

\section{Exactness of the first Lasserre relaxation}
\label{sec:exact}

\begin{assumption}[SOS convexity]\label{ass:lfsos}
	The function $C\mapsto \tilde{W}(C)$ is SOS convex, i.e. its Hessian matrix is SOS. 
\end{assumption}

SOS convexity is a notion introduced originally in \cite{helton-nie10} to study semidefinite representations of convex sets.
SOS convexity is stronger than convexity, even though it is challenging to find polynomials that are convex but not SOS convex, see \cite{AhmadiParrilo2012} for an explicit example of degree 8 in 3 variables, and \cite[Sec. 7.1.1]{nie23} for a recent survey.

Let us prove that, under Assumption \ref{ass:lfsos},  the first relaxation of the occupation-measure moment-SOS hierarchy \eqref{eq:SDP-occupation} is exact
for the COV problem \eqref{eq:cov} when the Dirichlet boundary condition is linear.
In that case, the variational problem is a computation of the quasiconvex envelope $W_{\mathrm{quasi}}(F)$.

\begin{theorem}[Exactness of the first relaxation for linear boundary data]\label{thm:first-relax-exact}
Under Assumptions \ref{ass:growth-two-sided}, \ref{ass:frame}, \ref{ass:lfsos} and for a given matrix $A\in\mathbb R^{n\times n}$, consider the COV problem
	\eqref{eq:cov} with Dirichlet condition $y_\partial(x)=Ax$ on $\partial\Omega$.
	Then the first (order $r=r_{\min}$) relaxation of the occupation-measure SDP \eqref{eq:SDP-occupation} is exact: its optimal value equals
	the true value of \eqref{eq:cov} and is given by the semidefinite projection formula
	\begin{equation}\label{eq:first-relax-value}
		J_{\mathrm{mom}}^{r_{\min}}
		\;=\;
		|\Omega|\,W_{\mathrm{quasi}}(A).
	\end{equation}
\end{theorem}

\begin{proof}
Let $y$ be an admissible pseudo-moment vector in \eqref{eq:SDP-occupation}, and let
	\[
	Y_1 := \ell_y(Z) \in \R^{n\times n},
	\qquad
	Y_2 := \ell_y(Z^\top Z) \in \mathbb S_+^n.
	\]
	In the linear Dirichlet case $y_\partial(x)=Ax$, the (first degree) Stokes constraints appearing in \eqref{eq:SDP-occupation}
	imply
	\begin{equation}\label{eq:mass-mean-grad}
		y_0 = |\Omega|,
		\qquad
		Y_1= |\Omega|\,A.
	\end{equation}
	\revision{Indeed, the coordinate test functions $x_j$ and $y_i$ in the Stokes constraints fix, respectively, the total mass and the entries of the mean gradient.}
	The moment matrix constraint in \eqref{eq:SDP-occupation} implies that the matrix
	\[
		\begin{pmatrix}
			y_0 & Y_1\\
			Y^\top_1 & Y_2
		\end{pmatrix}
		\succeq 0
	\]
	is positive semidefinite. Since $y_0>0$, taking the Schur complement yields
	\[
		Y_2 \;\succeq\; \frac{1}{y_0}\,Y^\top_1 Y_1.
	\]
	Using \eqref{eq:mass-mean-grad} we obtain
	\begin{equation}\label{eq:C-dom}
		\frac{1}{|\Omega|}\,Y_2 \;\succeq\; A^\top A.
	\end{equation}
	Hence there exists $P\succeq 0$ such that
	\begin{equation}\label{eq:C-FP}
		\frac{1}{|\Omega|}\,Y_2 = A^\top A + P.
	\end{equation}

Recall that $W(Z)=\widetilde W(Z^\top Z)$ with $\widetilde W$  SOS convex by Assumption \ref{ass:lfsos}. 
\revision{The extended Jensen's inequality proposed in \cite[Thm. 2.6]{lasserre09} (see also \cite[Sec. 7.1.2]{nie23}) states that, for any normalized positive pseudo-moment functional $L$ and any SOS-convex polynomial $p(s)$, one has $L(p(s))\geq p(L(\mathrm{s}))$. Here it is applied to the matrix variable $Z^\top Z$ with $L=\ell_z/|\Omega|$
} to establish that 
\[
\frac{1}{|\Omega|}\,\ell_z\bigl(W\bigr) = 
\frac{1}{|\Omega|}\,\ell_z\bigl(\widetilde W(Z^\top Z)\bigr)
\ \ge\
\widetilde W \left(\frac{1}{|\Omega|}\,\ell_z(Z^\top Z)\right)
=
\widetilde W(\frac{1}{|\Omega|}\,Y_2)
=
\widetilde W(A^\top A + P).
\]
Therefore,
\begin{equation}\label{eq:lower-bound-pseudo}
\ell_z(W) \ge\ 
|\Omega|\,\widetilde W(A^\top A+P).
\end{equation}
Taking the infimum over all feasible pseudo-moment sequences $z$ gives
\begin{equation}\label{eq:first-relax-lb-pseudo}
J_{\mathrm{mom}}^{r_{\min}}
\ \ge\
|\Omega|\inf_{Q\succeq 0}\widetilde W(A^\top A+Q).
\end{equation}
The reverse inequality follows immediately since, by construction, the first relaxation \eqref{eq:SDP-occupation} is a relaxation of the linear problem \eqref{eq:LP-occupation} that is itself a relaxation of the original problem.
\end{proof}

Theorem~\ref{thm:first-relax-exact} is a linear boundary data result.
For general (nonlinear) Dirichlet data $y_\partial$, the first order relaxation typically cannot encode sufficient compatibility
	information beyond the mass and mean-gradient constraints, so exactness at the first level should not be expected without
	additional structure.

\section{Minimizer extraction}\label{sec:extraction}

\subsection{Extraction from the occupation measure}
The goal of this section is to describe how one can obtain an approximate solution to the relaxed variational problem~(\ref{eq:relaxed}) from the solution to the SDP relaxation of order $r$. First, one has to observe that the solution to the infinite dimensional LP is a priori not guaranteed to be supported on a single minimizer (or a superposition thereof) due to the vectorial nature of the problem (see \cite[Section 3]{krzgap}). Fortunately, the convexity of $W_{\text{quasi}}$ allows us to extract one minimizer of (\ref{eq:relaxed}) as the conditional barycenter $\bar y(x)$ of the measure $\mu(x,y,Z)$ defined in~\eqref{eq:y_bar}. Numerically, we approximate this conditional barycenter using the $\mathscr L_2$ definition of conditional expectation with a polynomial ansatz for $\bar  y$. Concretely, we solve
\begin{equation} \label{eq:extract}
\arg\min_{p\in\mathbb{R}[x]_r} \int \|y-p(x)\|^2\,d\mu(x,y,Z),
\end{equation}
which is equivalent to solving the normal equations
\[
\left(\int \phi(x)\phi(x)^\top\,d\mu(x,y,Z)\right)\,c
=
\int \phi(x)\,y\,d\mu(x,y,Z),
\]
where $\phi(x) \in \mathcal B_r$ is a polynomial basis vector of degree $r$ and $c \in \R^N$ is the unknown coefficient vector, where $N = \binom{n+r}{r}$. The barycenter $\bar y(x)$ can be approximated by $p(x)= c^T \phi(x)$, where $c$ is an optimal solution to the normal equations. A similar procedure, using $\mathscr L^2$ regression, was used for (unconditional) density approximation in~\cite{Henrion2014density}. \revision{The connection of the extracted polynomial and the deformation $\bar y$ defined in \eqref{eq:y_bar} is precisely formulated by the following proposition.}

\revision{
\begin{proposition}[Convergence of the barycentric extraction]\label{prop_convergence_1}
Let $p\geq 2$, let $\mu$ be a feasible in \eqref{eq:LP-occupation}, and let
\[
    d\mu(x,y,Z)=d\nu_x(y,Z)\,dx,
    \qquad
    \bar y(x):=
    \int_{\mathbb{R}^n\times\mathbb{R}^{n\times n}}
    y\,d\nu_x(y,Z)
\]
be the barycentric deformation defined in the proof of Theorem~3. For
every $r\in\mathbb{N}$, let $p_r\in\mathbb{R}[x]_r$ be a minimizer
of \eqref{eq:extract}. Then
\[
    p_r\longrightarrow\bar y
    \qquad\text{strongly in }\mathscr L^2(\Omega;\mathbb{R}^n)
\]
as $r\to\infty$.
\end{proposition}

\begin{proof}
First, the lower $p$-growth bound on $W$ and the finite energy of $\mu$
imply that
\[
    \int |Z|^p\,d\mu<\infty.
\]
The Stokes constraints then imply
\[
    \int |y|^p\,d\mu<\infty.
\]
Since $p\geq2$ and $\mu$ has finite mass, it follows that $\bar y\in \mathscr L^2(\Omega;\mathbb{R}^n)$.

For any polynomial map $p$, disintegration of $\mu$ and the definition
of $\bar y$ yield

\begin{align*}
    \int \|y-p(x)\|^2\,d\mu
    &=
    \int_\Omega\int
    \|y-\bar y(x)+\bar y(x)-p(x)\|^2
    \,d\nu_x(y,Z)\,dx \\
    &=
    \int_\Omega\int \|y-\bar y(x)\|^2
    \,d\nu_x(y,Z)\,dx
    +
    \|\bar y-p\|_{L^2(\Omega)}^2,
\end{align*}

where the mixed term vanishes because
\[
    \int \bigl(y-\bar y(x)\bigr)\,d\nu_x(y,Z)=0
\]
for almost every $x\in\Omega$. Note that the first term is independent of $p$, $p_r$ is the minimizer of the second term and polynomials are dense in $\mathscr L^2(\overline\Omega; \R^n)$. This implies that for every $\epsilon > 0$, there exists $r$ sufficiently large such that
\[
\|\bar y - p_r\|_{L^2} < \epsilon,
\]
which proves the statement.
\end{proof}
}

\revision{We would like to emphasize that this procedure extracts the barycentric deformation $\bar y(x)$. This means that possible oscillations in the gradient are averaged and as a result, when the original problem has no minimizer, there might be a discrepancy between the true infimum of the original problem and the the energy evaluated at $\bar y (x)$. }

\revision{
\subsection{Boundary conforming approximate extraction}
To obtain a deformation satisfying the prescribed boundary condition exactly, one may incorporate the boundary condition directly into the polynomial ansatz of the previsous subsection. Namely, let $b(x)$ be a polynomial vanishing on $\partial\Omega$, and approximate the conditional barycenter by
\[
p(x)=y_{\partial}(x)+b(x)c^\top\phi(x),
\]
where $\phi(x)$ is a polynomial basis of a suitable degree. The coefficient vector $c$ is then obtained by
\begin{align}\label{eq:minimizer_boundary}
\min_{c^{\top}\phi(x)}\int \|y-y_{\partial}(x)-b(x)c^\top\phi(x)\|^2d\mu(x,y,Z),
\end{align}
which again reduces to a system of linear normal equations involving only moments available from the SDP solution. By construction, $p=y_{\partial}$ on $\partial\Omega$, independently of the finite-order enforcement of the Stokes constraints. An attractive feature of this approximation compared to the previous one is that evaluating the original stored energy at $p$ yields an upper bound on the infimum of the variational problem. However, a setback of this method compared to the previous one is that multiplying the basis $\phi(x)$ by a polynomial vanishing on $\partial \Omega$ artificially decreases the degree of the approximation, as the degree of $b(x)\phi(x)$ cannot exceed the relaxation order. Because of that, the practical applicability of this method is limited, although it has some nice properties.

The following result is analogous to Proposition~\ref{prop_convergence_1} and it is stated without proof as the proof strategy would be the same as in proving Proposition \ref{prop_convergence_1}.

\begin{proposition}[Convergence of the boundary conforming extraction]\label{prop_convergence_2}
Let $p\geq 2$, let $\mu$ be feasible in \eqref{eq:LP-occupation}, and let
\[
    d\mu(x,y,Z)=d\nu_x(y,Z)\,dx,
    \qquad
    \bar y(x):=
    \int_{\mathbb{R}^n\times\mathbb{R}^{n\times n}}
    y\,d\nu_x(y,Z)
\]
be the barycentric deformation defined in the proof of Theorem~3. Let
$b\in\mathbb{R}[x]$ satisfy $b(x)=0$ on $\partial\Omega$ and
$b(x)\neq 0$ almost everywhere in $\Omega$. For every $r\in\mathbb{N}$, let $p_r\in \R[x]_r$ be a minimizer of \eqref{eq:minimizer_boundary}. Then,
\[
    p_r\longrightarrow\bar y
    \qquad\text{strongly in }\mathscr L^2(\Omega;\mathbb{R}^n)
\]
as $r\to\infty$.
\end{proposition}

}

\section{Examples}\label{sec:examples}
In this section, we numerically solve the occupation measure relaxation \eqref{eq:LP-occupation} on the isotropic Saint Venant-Kirchoff (SVK) energy function with zero Poisson ratio, and an anisotropic variation, which are all convex in the Cauchy-Green strain.  \revision{We alternatively used GloptiPoly \cite{henrion_gloptipoly_2009} and Python} to model the optimization problem and MOSEK \cite{mosek} \revision{and its Python Fusion interface} to solve the SDP problems arising from the relaxations of the infinite-dimensional LP, which were introduced in Section \ref{sec:momsos}. We use the minimizer extraction procedure described in Section~\ref{sec:extraction}. The apriori bound $R$ on the $(y,Z)$ variables, as introduced in Section \ref{sec:momsos} was selected such that increasing the bound further did not change the optimum. In the case of a linear boundary condition, we know that the choice of $R$ is sufficient by obtaining the analytical value.

Consider the SVK energy function
\begin{equation}\label{eq:svk}
W(F) = \frac{\lambda}{2} (\tr E)^2 + \mu \tr(E^2)
\end{equation}
where $E = \frac{1}{2}(F^T F- I_2)$ is the Green-Lagrange strain tensor and $\lambda, \mu \geq 0$ are the Lam\'e parameters. The value $\nu = \frac{\lambda}{2(\lambda+\mu)}$ is the Poisson ratio, which measures how much thinner the material gets relative to how much longer it gets.

\subsection{SVK with zero Poisson ratio}

\begin{figure}[ht]
	\centering
	\includegraphics[width=0.7\textwidth]{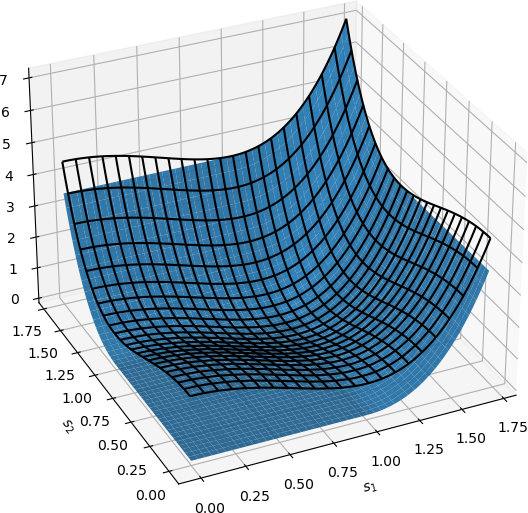}
	\caption{The black wired frame is the nonconvex energy
		$W(F)=|F^T F-I_2|^2$ in the plane of singular values
		$(s_1,s_2)$ of $F$. The blue surface below is the quasiconvex envelope $W_{\mathrm{quasi}}$, which is equal to $W$ when $\min(s_1,s_2) \geq 1$.
	\label{fig:cauchy-green-surface}}
\end{figure}
	
When $\nu = \lambda = 0$ and $\mu=4$, the energy function simplifies to
	\[
	W(F)=|F^T F-I_2|^2,
	\qquad F\in\R^{2\times2}
	\]
already mentioned in Section \ref{sec:cauchygreen}. It satisfies the growth Assumption \ref{ass:growth-two-sided} and the frame indifference Assumption \ref{ass:frame}.	
	Here
	\[
	\widetilde W(C)=|C-I_2|^2
	\]
	is convex and quadratic in
	$C\in\Sbb^2_+$, satisfying the convexity Assumption \ref{ass:convex}. Denoting the singular values of $F$ by $s_1\geq s_2\geq 0$, it holds
	\[
	W(F) = (s^2_1-1)^2+(s^2_2-1)^2.
	\]
	The semidefinite projection formula \eqref{eq:LDR} of Theorem \ref{thm:LDR} yields
	\[
	W_{\mathrm{quasi}}(F) = \inf_{S \in \Sbb^2_+} |S- (I_2 -F^T F)|^2
	\]
	which is the squared Frobenius distance from $I_2 -F^T F$ to the semidefinite cone $\Sbb^2_+$, see Figure \ref{fig:cauchy-green-surface}. The orthogonal projection onto 
	$\Sbb^2_+$ is given by spectral truncation:
	\begin{equation}\label{spectral}
	W_{\mathrm{quasi}}(F)  = \max(s^2_1-1,0)^2 + \max(s^2_2-1,0)^2
	\end{equation}
	so that $W(F)=W_{\mathrm{quasi}}(F)$ when $\min(s_1,s_2)\geq 1$, see Figure \ref{fig:cauchy-green-surface}.

\subsubsection{Linear boundary conditions}

Consider the following calculus of variations problem
\begin{align}
\label{ex_simple_original}
    J = &\inf_{y \in \mathscr W^{1, 4}(\Omega)} \int_{\Omega}|F^T(x)F(x) - I|^2dx \\
    \text{s.t.: }&\left.y\right|_{\partial \Omega} = Ax,\notag\\
    & y(x) \in \R^2, \quad F(x) = \nabla y(x) \in \R^{2\times 2} \notag,
\end{align}
where $A$ is a given full rank matrix and $\Omega = [0,1] \times [0,1]$. We choose $y \in \mathscr W^{1, 4}(\Omega)$, because the degree of the polynomial integrand is 4. The occupation measure relaxation of problem \eqref{ex_simple_original}, which corresponds to problem \eqref{eq:LP-occupation}, is
\begin{align}
\label{ex_simple_measure}
        J_{\text{lin}} := &\inf_{\mu \in\mathcal{M}_+(\Omega \times \R^2 \times \R^{2\times 2})}\int |Z^TZ - I|^2\,d\mu(x,y,Z) \\
    \text{s.t.: } & \int 
		\mathcal D\phi \, \mu 
		=
		\int 
		\phi \cdot n \, \mu_\partial, \quad \forall \phi \in \Phi. \notag
\end{align}
Problem \eqref{ex_simple_measure} can be directly solved with the moment-SOS hierarchy.

Problem \eqref{ex_simple_original} has an analytical solution, as described in the previous section. We can compare the sequence of optima of the SDP problems corresponding to solving \eqref{ex_simple_measure} with increasing relaxation order (or equivalently, increasing degree of polynomial test functions in the constraint $\mathcal D\phi \, \mu =\int \phi \cdot n \, \mu_\partial$) with this analytical optimum. However, since function $\tilde{W}(C)$ is convex quadratic, its Hessian is positive semidefinite, therefore Assumption~\ref{ass:lfsos} is satisfied, and from Theorem \ref{thm:first-relax-exact}, we know that the first relaxation is already exact in the case of linear boundary condition.

Numerically, for a matrix 
\[
A = \begin{pmatrix}
    1.15& 0.65\\ 0.65& 1.15
\end{pmatrix},
\]
the numerical value obtained at the first relaxation matches the analytical optimum \revision{$5.01760$} at 7 significant digits. The barycenter deformation $\bar y(x)$ extracted on a wireframe grid ${x_i}$ (with 7 horizontal and vertical lines, each consisting of 80 points) on $\Omega$, is shown in Figure \ref{fig:simple}. The black dots are the samples of the wireframe in $\Omega$ and the red dots represent the extracted $\bar y(x)$ of the black grid. The black filled line is the original boundary of $\Omega$ and the red filled line is the boundary condition $y(x) = Ax$ on the boundary of $\Omega$.
\begin{figure}[h!]
    \centering
    \includegraphics[width=0.7\linewidth]{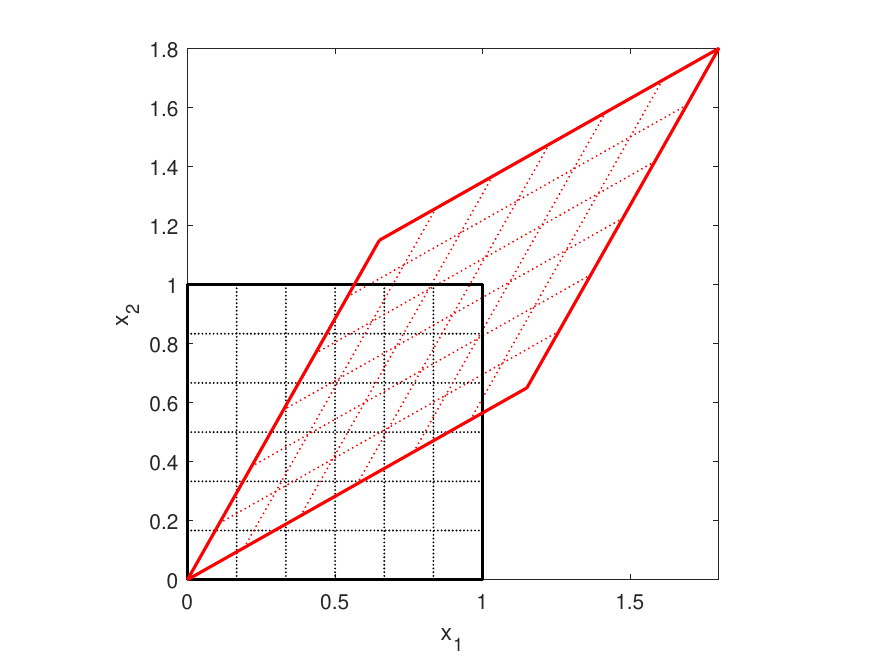}
    \caption{Original domain $\Omega$ (black frame) and deformed domain $\bar y(\Omega)$ (red frame) for the SVK energy with zero Poisson ratio and linear boundary conditions.}
    \label{fig:simple}
\end{figure}

\subsubsection{Quadratic boundary conditions}

Let us now focus on the quadratic boundary conditions. The problem reads
\begin{align}
\label{ex_simple_quad_original}
    J = &\inf_{y \in \mathscr W^{1, 4}(\Omega)} \int_{\Omega}|F^TF - I|^2dx \\
    \text{s.t.: }&\left.\begin{pmatrix}
        y_1\\y_2
    \end{pmatrix}\right|_{\partial \Omega} = Ax + \begin{pmatrix}
        x^TB_1x\\x^T B_2x
    \end{pmatrix},\notag\\
    & y(x) \in \R^2, \quad F = \nabla y \in \R^{2\times 2} \notag,
\end{align}
where $B_1, B_2$ are given matrices, $A$ is a given full rank matrix and $\Omega = [0,1]^2$. The occupation measure relaxation of problem \eqref{ex_simple_quad_original}, which corresponds to problem \eqref{eq:LP-occupation}, is
\begin{align}
\label{ex_simple_quad_measure}
        J_{\text{lin}} := &\inf_{\mu \in\mathcal{M}_+(\Omega \times \R^2 \times \R^{2\times 2})}\int |Z^TZ - I|^2\,d\mu \\
    \text{s.t.: } & \int 
		\mathcal D\phi \, \mu 
		=
		\int 
		\phi \cdot n \, \mu_\partial, \quad \forall \phi \in \Phi. \notag
\end{align}
Problem \eqref{ex_simple_measure} can be directly solved with the moment-SOS hierarchy. However, for the quadratic boundary condition, we do not expect the first relaxation of \eqref{ex_simple_quad_measure} to be exact and thus, each relaxation of the hierarchy provides us with a lower bound for the optimum. For a particular choice of matrices
\begin{align*}
    A = \begin{pmatrix}
        0.8& 0\\0.1& -1
    \end{pmatrix}, \qquad B_1 = \begin{pmatrix}
        1.1 & 0.25\\ 
           0.25 & 1
    \end{pmatrix}, \qquad B_2 = \begin{pmatrix}
        0 & 0\\ 0& 0
    \end{pmatrix},
\end{align*}
we obtain the nondecreasing sequence of lower bounds reported at the second row of Table \ref{tab:bounds}. Figure \ref{fig:quad_simple} shows the extracted barycenter deformation for degrees 4 and 8, corresponding to relaxation orders 2 and 4, respectively. This figure shows that at degree $4$, the boundary condition is not satisfied and as the degree increases, the boundary condition becomes satisfied with greater precision.

\begin{table}[h]
\centering
\begin{tabular}{c|ccc}
relaxation order & 2 & 3 & 4 \\ \hline
lower bound & 29.813 & 32.139 & 32.703 \\
barycentric approximation & \revision{35.984} & \revision{31.923} & \revision{33.216}
\end{tabular}
\caption{Approximations of the optimum (5 significant digits), for various relaxation orders.}
\label{tab:bounds}
\end{table}

\begin{figure}[h!]
    \centering
    \begin{subfigure}[t]{0.5\textwidth}
        \centering
        \includegraphics[width=1\linewidth]{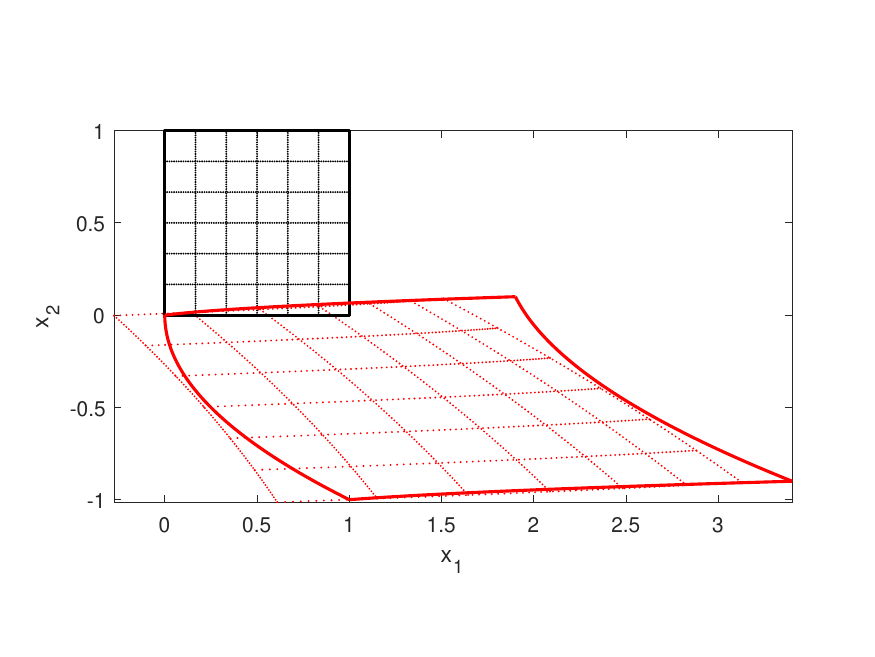}
        \caption{Relaxation order 2}
    \end{subfigure}~
    \begin{subfigure}[t]{0.5\textwidth}
        \centering
        \includegraphics[width=1\linewidth]{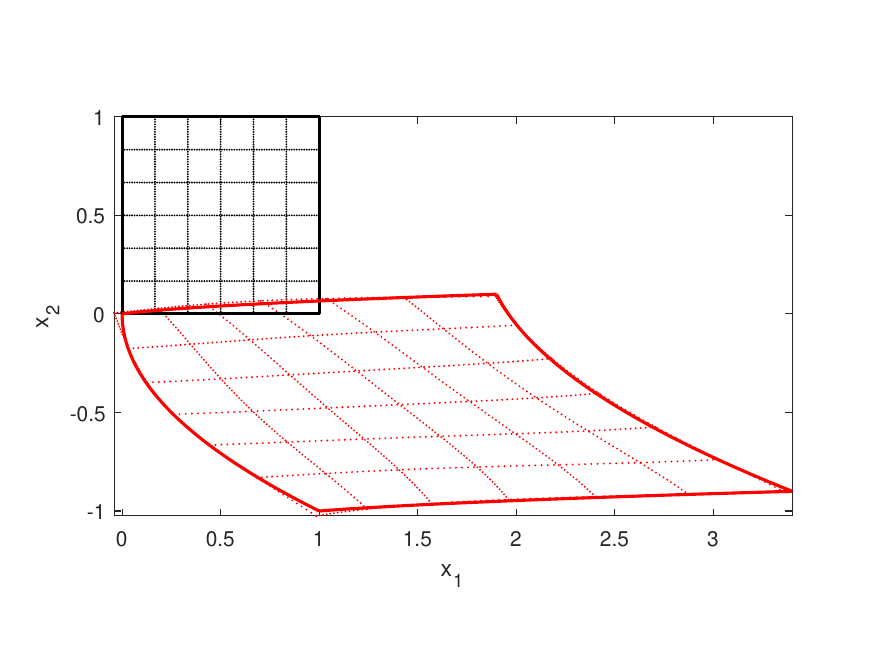}
        \caption{Relaxation order 4}
    \end{subfigure}
    \caption{Original domain $\Omega$ (black frame), deformed domain $\bar y(\Omega)$ (red frame) for the simple stored energy density with quadratic boundary condition}
    
\label{fig:quad_simple}
\end{figure}

Since the barycenter deformation $\bar y(x)$ should be optimal for the quasiconvexified problem \eqref{eq:relaxed}, and we have an analytical formula for the quasiconvex envelope \eqref{spectral}, we can plug in the approximate $\nabla \bar y(x)$ to the quasiconvexified problem \eqref{eq:relaxed} and obtain an approximate optimum by \revision{exact quadrature}. For approximate barycenter deformation of increasing degrees, we get the approximate values reported at the bottom row of Table \ref{tab:bounds}. Note that this computation suffers from multiple numerical errors. Namely, the barycenter $\bar y(x)$ is only approximate. \revision{The discrepancy between the lower bounds obtained by the SDP solver and the barycentric approximation indicates slow convergence of the deformation as the relaxation order increases.}

\subsection{SVK energy}
Introducing the matrix
\[
X = \left(\begin{array}{cc} X_{11} & X_{12} \\ X_{12} & X_{22} \end{array}\right) := 2 E = F^T F - I_2,
\]
the SVK energy function can be expressed as
\begin{equation*}
    W(F)
    =
    a(F)^T D \,a(F),
    \qquad
    a(F)
    =
    \begin{pmatrix}
        X_{11} \\
        X_{22} \\
        2X_{12}
    \end{pmatrix}
\end{equation*}
where $D$ is a positive definite stiffness matrix. The energy minimization example is then as follows
\begin{align}
\label{ex_simple_general}
    J = &\inf_{y \in \mathscr W^{1, 4}(\Omega)} \int_{\Omega} a(\nabla y(x))^T D \, a(\nabla y(x)) dx \\
    \text{s.t.: }&\left.y\right|_{\partial \Omega} = Ax,\notag\\
    & y(x) \in \R^2, \quad \nabla y(x) \in \R^{2\times 2} \notag,
\end{align}
where $A$ is a given full rank matrix and $\Omega = [0, 1]^2$.

\revision{Here we distinguish two cases. First,} when the energy function $W(F)$ is isotropic, i.e. when the stiffness matrix $D$ has the form \revision{
\begin{equation*}
\begin{pmatrix}
\lambda + 2\mu & \lambda        & 0 \\
\lambda        & \lambda + 2\mu & 0 \\
0              & 0              & \mu
\end{pmatrix}
\end{equation*}}
for some Lam\'e parameters $\lambda, \mu \geq 0$, an explicit formula for the quasiconvex envelope of the SVK energy was derived in \cite{LeDretRaoult1995b}. The infimum of problem \eqref{ex_simple_general} can be obtained by evaluating the quasiconvex envelope $W_{\text{quasi}}(F)$ at $F = A$ and thus, we can plug $A$ into the expression for the quasiconvex envelope to obtain the precise analytical value. \revision{We focus on the second case, that is,} the stiffness matrix $D$ is positive definite but not necessarily \revision{describable by Lam\'e parameters}, the function $W(F)$ is anisotropic in general, a closed form expression is not available.

For the anisotropic case, let us solve \eqref{ex_simple_general} when
\[
D = \begin{pmatrix}
    20 & 2 & 0\\
    2 & 5 & 0\\
    0 & 0 & 3\\
\end{pmatrix}.
\]
In this case, we can still compute the value for the linear boundary condition case from the formula of Theorem \ref{thm:LDR}. Indeed, we can numerically solve the following SDP
\begin{equation}
\label{little_sdp}
    \inf_{P \succeq 0}\widetilde W(A^TA + P),
\end{equation}
where $\widetilde W$ is defined in \eqref{eq:w_tilde},
which provides us with the numerical approximation of the optimal value. This is an SDP with quadratic cost, which can be reformulated as a linear SDP with second order cone constraint, and solved by MOSEK. For the choice
\[
A = \begin{pmatrix}
    0.9& -0.3\\ 1.2& 0.6
\end{pmatrix},
\]
the approximated value of the optimum from the SDP \eqref{little_sdp} is 32.383260, matching at 8 digits the value of the first relaxation. The approximate barycenter $\bar y(x)$ of the same wireframe grid as in previous examples is shown in Figure \ref{fig:anisotropic}.
 
\begin{figure}[h!]
    \centering
    \includegraphics[width=0.7\linewidth]{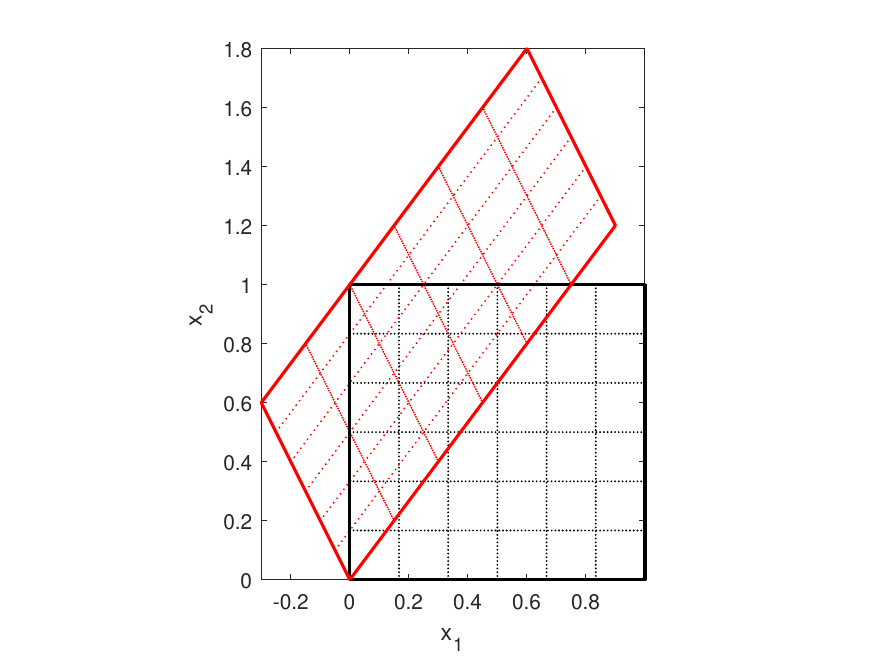}
    \caption{Original domain $\Omega$ (black frame), deformed domain $\bar y(\Omega)$ (red frame) for the anisotropic energy density.}
    \label{fig:anisotropic}
\end{figure}

	\section{Conclusions and perspectives}
	\label{sec:conclusion}
	
	We have shown how semidefinite optimization can be used to solve a class
	of non-convex stored energy minimization problems in nonlinear
	elasticity. Under frame indifference and convexity with respect to the
	Cauchy-Green strain tensor, the calculus of variations problem can be solved numerically with the Lasserre hierarchy with convergence guarantees. When the strain energy is polynomial SOS convex, and the boundary condition is linear, we established that the first
	relaxation of the hierarchy is exact. Consequently, the
	original non-convex calculus of variations problem with linear boundary
	data can be solved via a single small semidefinite program in the
	strain variables.
	
	The frame indifference exploited here is a
	manifestation of symmetry. So this work can be seen as a contribution to the literature on polynomial optimization which focuses on explicit
	reduction of the moment matrices and orbit decomposition, see \cite{Augier2024} and references therein.
	
	Several directions for future work remain open:
	\begin{itemize}
		\item \emph{Polyconvex energies.} In many models of nonlinear elasticity
		the stored energy is polyconvex rather than convex in the strain
		tensor. Extending the present no-relaxation-gap and exactness results to
		polyconvex integrands would require an appropriate moment formulation in
		the space of minors of the gradient matrix $F$.
		
		\item \emph{Quasiconvex but non-convex strain energies.} When Assumption
		\ref{ass:convex} fails but the energy $W$ remains quasiconvex, one still has no
		relaxation gap between the calculus of variations problem and its gradient Young measure
		formulation. Identifying conditions under which the Lasserre hierarchy converges, or the first or finitely
		many levels are exact in this broader
		setting is an interesting question.

	\end{itemize}
	
	Ultimately, a central objective is to establish general no-relaxation-gap
	results between the linear problem on occupation measures and the
	original stored energy minimization problem for polyconvex and quasiconvex integrands,
	and to identify structural conditions guaranteeing finite convergence of
	the moment-SOS hierarchy in this infinite-dimensional setting.

\section*{Acknowledgements}

This work was funded by the European Union/M\v{S}MT \v{C}R  under the ROBOPROX project (reg.~no.~CZ.02.01.01/00/22 008/0004590), by the Grant Agency of the Czech Technical University in Prague (grant No. SGS25/145/OHK3/3T/13) and by the CNRS Informatics International Research Project OPTIROB. \revision{The authors are grateful to two anonymous reviewers who provided constructive feedback on this work.}
The authors acknowledge the use of AI for assistance with brainstorming ideas, mathematical development, coding and drafting the manuscript. The final content, analysis and conclusions remain the sole responsibility of the authors.

\end{document}